\documentclass[a4paper,10pt]{amsart}
\usepackage{amssymb}
\usepackage{graphicx}
\newtheorem{theorem}{Theorem}
\newtheorem{corollary}[theorem]{Corollary}
\newtheorem{definition}{Definition}
\newtheorem{lemma}[theorem]{Lemma}
\newtheorem{proposition}[theorem]{Proposition}
\newtheorem{remark}{Remark}

\newcommand{\NN}{{\rm\bf N}}
\newcommand{\ZZ}{{\rm\bf Z}}
\newcommand{\RR}{{\rm\bf R}}
\newcommand{\EU}{{\rm\bf S}}

\newcommand{{\markov}}{T}

\begin{document}
\title[Attractors in complex networks]{Attractors in complex networks}


\author[Alexandre A. P. Rodrigues]{Alexandre A. P. Rodrigues \\ Centro de Matem\'atica da Universidade do Porto \\ 
and Faculdade de Ci\^encias da Universidade do Porto\\
Rua do Campo Alegre 687, 4169--007 Porto, Portugal}
\address[A. A. P. Rodrigues]{Centro de Matem\'atica da Universidade do Porto\\
and Faculdade de Ci\^encias da Universidade do Porto\\ \\
Rua do Campo Alegre 687, 4169--007 Porto, Portugal}
\email[A.A.P.Rodrigues]{alexandre.rodrigues@fc.up.pt}

\date{\today }

\begin{abstract}
In the framework of the generalized Lotka Volterra model, solutions representing multispecies sequencial competition can be predictable with high probability. In this paper, we show that it occurs because the corresponding ``heteroclinic channel'' forms part of an attractor. We prove that, generically, in an attracting heteroclinic network involving a finite number of hyperbolic  and non-resonant saddle-equilibria whose linearization has only real eigenvalues, the connections corresponding to the most positive expanding eigenvalues form a part of an attractor (observable in numerical simulations).
  \end{abstract}
\maketitle

\textbf{Heteroclinic networks are much studied because they organise dynamics in which saddle-type equilibria are visible. Heteroclinic networks for differential equations can be accompanied by complicated nearby dynamics. Apart from topological descriptions of this complicated dynamics, a relevant question is the fate of typical, say for a full Lebesgue measure of initial conditions. This question has not always attracted appropriate attention. 
Asymptotic stability is well understood for a large class of heteroclinic
cycles in the symmetric context. However, in many cases this is not satisfactory, because a Lyapunov unstable heteroclinic cycle may
attract a set of large measure of initial states and may be observable in numerical simulations. }

\textbf{In the context of the generalized Lotka-Volterra system, the authors of \cite{ATHR, MTHAR} studied heteroclinic networks whose linearization at the equilibria has only real eigenvalues. They claim that there is a preference of the system to evolve in the strongest direction with and without the influence of noise. Rigorous results about the $\omega$-limit sets for these networks are going to be needed in the very near future to discuss the stability of networks that appear in the context of \emph{winnerless competition} and \emph{synchronization into clusters} that appear in several biological systems. }

\textbf{In the present paper, motivated by partial results in \cite{AC} and supported by numerical simulations of \cite{ATHR, MTHAR},
we show that, within a heteroclinic network involving hyperbolic equilibria whose linearization has only real eigenvalues, generically the connections corresponding to the most positive expanding eigenvalues form a part of an attractor. }

\section{Introduction}
Heteroclinic cycles are flow-invariant sets consisting of finitely many
equilibria and trajectories connecting them in a cyclic fashion.  They appear in several applications and are useful to the study of intermittent dynamics: a trajectory near a heteroclinic cycle will spend a long time in a neighbourhood of an equilibrium, before rapidly switching along a connection towards another equilibrium, where it stays again for a long period of time. Such behaviour is displayed, for example, by the geomagnetic field: its fast, unpredictable reversals of polarity are followed by long periods in a stationary mode. Many authors claim the presence of heteroclinic networks in the equations that describe geodynamic processes -- see \cite{MPR, Rodrigues2} and references therein.

Another example occurs in population dynamics and has been treated in \cite{ATHR, HS}:
as a model for competition between three or more species, Lotka-Volterra equations can 
possess attracting heteroclinic
networks between equilibria in which there is only one winner. It seems for some time as if a unique species wins the competition and all others become
extinct, before its density suddenly drops and its dominant state is taken by another
species.

With dos Reis \cite{dosReis} and Guckenheimer
and Holmes \cite{GH},  it has been discovered that heteroclinic cycles and networks may be structurally stable in an equivariant context.
Subsequently, their study gained importance. Placing the
connections in flow-invariant subspaces, where they are of saddle-sink type, and restricting
to perturbations that respect these symmetries, heteroclinic cycles becomes persistent invariant sets: if
the invariant subspaces are not destroyed under small smooth perturbations, the cycle persists. In \cite{HomS, krupa}, the authors give a
comprehensive overview of early results on robust heteroclinic cycles.

The discovery of robustness ignited an interest on heteroclinic cycles in the nineties. 
Necessary and sufficient conditions 
for asymptotic stability of different types of cycles were derived in \cite{KM1, KM2}. However, being more complex in structure than a single hyperbolic equilibrium, 
heteroclinic cycles exhibit more complex stability features than the Lyapunov dichotomy
\emph{asymptotic stability vs instability},
in the sense that everything except for a set of
zero Lebesgue measure might remain or leave a small neighbourhood of the cycle.

In \cite{Melbourne}, Melbourne presents a paradigmatic example of a cycle that is
a non-asymptotically stable attractor: it attracts a set of positive Lebesgue measure of initial conditions, which is not a full neighbourhood of the cycle. The phenomenon becomes more surprising when several
cycles are joined together to form a heteroclinic network. A cycle belonging to a network has at least one positive transverse eigenvalue and therefore cannot be asymptotically stable.  
However, in many cases, a dominant cycle may be ``more stable'' than the others
in the sense that it may be the $\omega$-limit set of a large set (in terms of measure) of initial conditions in a neighbourhood of
the network, even though the other cycles may attract infinitely many solutions. 
Brannath \cite{Brannath} provides simple examples of these networks on simplices, where a single cycle attracts a large set of initial conditions.
\medbreak

\medbreak
In  \cite{ATHR, MTHAR}, the authors predict the behaviour of complex multi-agent systems, based on the \emph{Winnerless Competition Principle} that induces robust dynamics in complex heteroclinic networks. See also \cite{Hou}.  Motivated by \cite{KS} and \cite{ATHR, MTHAR}, we may ask:
\begin{enumerate}
\item What are the limit sets of the network? According to \cite{Milnor}, are they likely limit sets? 
\item In the case of a stable heteroclinic network which is a union of the one-dimensional connections there are only a finite number of possible $\omega$-limit sets? 
\end{enumerate}

For general (non-symmetric) attracting heteroclinic networks, we expect the presence of essentially asymptotically stable subnetworks. More precisely, if $\Gamma$ is an attracting heteroclinic network involving a finite number of equilibria with only real eigenvalues, then the connections corresponding to the most positive expanding eigenvalues of the linearization determine a possibly smaller attractor.

\subsection*{Key ideas of the proof}
The key steps of this paper are a combination of the following ideas:
\begin{itemize}
\item start with an attracting heteroclinic network whose linearization at the hyperbolic equilibria has real and non-resonant eigenvalues;
\item at each equilibrium point, its strong unstable manifold corresponds to a heteroclinic connection to the next equilibrium point;
\item in the same spirit of Deng's Strong Lambda Lemma, we show that, near each equilibrium, the set of initial points that follow the strong unstable manifold  is (locally) the complement of a wedge and thus has asymptotically full Lebesgue measure;
\item the complement of wedges are generically send (under the transition map) into the complement of other wedges with the same property.  This process can be repeated \emph{ad infinitum}.
\end{itemize}
\bigbreak
\subsection*{Framework of the paper}
The goal of this paper is to prove that, within a heteroclinic network whose nodes are hyperbolic equilibria whose linearization has only real eigenvalues, generically the connections corresponding to the most positive expanding eigenvalues form a part of an attractor (possibly not unique).  The main result is stated in Section~\ref{Section_main}, after collecting relevant notions in
Section~\ref{Preliminaries}. We also distinguish between various kinds of stability that have been developed along the last two decades.

Linearization techniques are used in Section~\ref{Linearization} to construct a local transition map around the equilibria and also a return map around the cycle. This section deals with the geometrical structures which allow us to get an understanding of the dynamics.
The proof of the main result is done in Section \ref{Proof of the Main Theorem}.
\medbreak
Throughout this paper, we have endeavoured to make a self
contained exposition bringing together all definitions and topics related to the proofs. We have stated short
lemmas and we have drawn illustrative figures to make the paper easily readable.

\section{Preliminaries}
\label{Preliminaries}
In this section, we present and discuss some definitions which will be useful throughout the article. Let $n,N\in \NN$ where $n \geq 4$, and $G$ is a compact region of $\RR^{n}$. We consider a system of ordinary differential equations
\begin{equation}
\label{general}
\dot{x}=f(x), \qquad x \in G\subset \RR^{n}
\end{equation}
where the vector field $f: G \rightarrow G$ is $C^2$ with flow given by the unique solution  $x(t)\mapsto\varphi(t,x_{0})\in G\subset \RR^{n}$. For $x \in G$, let us introduce the following norm in the set of vector fields:
$$
||f||_{C^1}= \sup_{x \in G} \left(||f||+ \left\|\frac{\partial f}{\partial x}\right\|\right),
$$
where $||.||$ represents the norm of the maximum in $\RR^{n}$.
Endowed with this norm, the set of vector fields becomes a Banach space, denoted by $\mathcal{X}$. A $\delta$-neighbourhood of the vector field ${f}\in\mathcal{X}$ is the set of all vector fields $\tilde{f}\in\mathcal{X}$ satisfying 
$
\left\| \tilde{f}-f \right\|<\delta.
$

\subsection{Heteroclinic terminology}
We start this subsection with a definition of heteroclinic cycle and network that suffices to our purposes. 
\medbreak
Given two equilibria $p_1$ and $p_2$, a  \emph{heteroclinic connection} from $p_1$ to $p_2$, denoted 
$[p_1\rightarrow p_2]$, is a connected flow-invariant  manifold contained in $W^{u}(p_1)\cap W^{s}(p_2)$, where $W^s(p)$ and $W^u(p)$ refer to the stable and unstable manifolds of the hyperbolic equilibrium $p$, respectively.  The dimension of the unstable manifold of an equilibrium $p$ will be called the \emph{Morse index} of $p$.
Throughout this work, we assume that the connections are one-dimensional. 
\medbreak

Let $\mathcal{S=}\{p_{j}:j\in \{1,\ldots,N\}\}$ be a finite ordered set of
equilibria. We say that there is a  {\em heteroclinic cycle }associated with $\mathcal{S}$ if 
$$
\forall j\in \{1,\ldots,N\},W^{u}(p_{j})\cap W^{s}(p_{j+1})\neq
\emptyset \pmod N.
$$
A \emph{heteroclinic network} is a connected union of heteroclinic cycles.  Hereafter, all equilibria (also called by nodes) will be hyperbolic. Depending on the geometry of their eigendirections, the eigenvalues of $df$ at the equilibria may be classified into four types: radial, contracting, expanding and transverse. We adress the reader to \cite{KM1} for this classification.

\subsection{Notions of Stability}
In order to gain a broader understanding of the dynamics associated with heteroclinic networks, it is essential to accurately distinguish between various kinds of non-asymptotic stability. 
 We discuss various forms of stability for compact sets $X\subset G\subset \RR^n$ that have been developed over the last three decades. 

\medbreak
Consider a compact subset $X\subset G$ that is invariant under the flow $\varphi(t, x_0)$, with $t\in \RR$ and $x_0\in G$, generated by the differential equation (\ref{general}). Following definitions in Milnor \cite{Milnor}, let $\mathcal{B}(X)$ be the basin of attraction of $X$ defined as: 
$$
\mathcal{B}(X) =\{x \in \RR^n : \omega(x)\subset X \},
$$
and let $\ell$ denote the $n$-dimensional Lebesgue measure. 
For $\varepsilon>0$, let $B_\varepsilon(X)$ be an open $\varepsilon$-neighbourhood of $X$. 
The $\varepsilon$-local basin of attraction of $X$ is defined as:
$$
\mathcal{B}_\varepsilon(X):=\left\{x \in \ B_\varepsilon(X):\quad \omega(x)\subset X \quad \text{and} \quad \varphi(t,x)\in B_\varepsilon(X), \quad \forall t\in \RR^+ \right\}.
$$

 We say that $X$ is an attractor if it attracts a set of positive measure \emph{i.e.}
if $\ell(\mathcal{B}(X)) > 0$. We now introduce different notions of stability. The following definition, due to Lyapunov \cite{Liapunov}, is well known in the literature. 

\begin{definition}
The set $X\subset G\subset \RR^m$ is called \emph{asymptotically stable} if for any neighbourhood $U$ of $X$ there is a neighbourhood $V$ of $X$ such that for all $x\in V$ we have $\varphi(t,x) \in U$, for all $t>0$ and $\omega(x) \subset X$.
\end{definition}

The well known distinction between \emph{asymptotic stability} and \emph{instability} is  too coarse  to study the stability of  heteroclinic cycles. This is particularly true for cycles that are
part of a bigger network: within a network, no single cycle can be asymptotically stable, as there
is always an invariant saddle with an unstable direction belonging to another cycle. There
might  be a dominant cycle that is observed for a large proportion of initial conditions, making
it more observable in terms of numerics, than the other cycles. 
Melbourne \cite{Melbourne} was
the first to give an explicit example of such an attractor and establish an intermediate type of stability:
\emph{essential asymptotic stability}.

\begin{definition}
The set $X\subset G\subset \RR^n$ is called \emph{essentially asymptotically stable}  if there is a set
$D \subset\RR^n$, so that for any neighbourhood $U$ of $X$ and any $a \in (0,1)$, there is a neighbourhood $V\subset U$ of $X$ such that:
\begin{enumerate}
\item for $x \in V \backslash D$ we have $\varphi(t,x) \in U$ for all $t > 0$, as well as $\omega(x) \subset X$, and 
\item $\frac{\ell(V\backslash D)}{\ell(V)}>a$.
\end{enumerate}
\end{definition}

The expression \emph{essential asymptotic stability} generated some confusion since various authors have used it with slightly different interpretations. There exist contradicting definitions in the literature: while that of \cite{Melbourne} is equivalent to
simply attracting any set of positive measure, in \cite{Brannath} the author defines \emph{essential asymptotic stability} in the same way as \emph{predominant asymptotic stability} defined by \cite{PA}. See also \cite{CL, Lohse} where the contradicting definitions have been detected and discussed. 
In order to precisely distinguish between different levels of instability we recall  the following definitions due to \cite{Brannath}. If $X \subset G\subset \RR^n$, let $\overline{X}$ stand for the topological closure of $X$.

\begin{definition}
The set $X\subset \RR^n$ is called \emph{asymptotically stable relative to the set} $N \subset \RR^n$ if $X \subset \overline{N}$ and for any neighbourhood $U$ of $X$, there is a neighbourhood $V$ such that for all $x\in V \cap N$ we have $\varphi(t,x) \in U$, for all $t>0$ and $\omega(x) \subset X$.
\end{definition}

\begin{definition}
The set $X\subset \RR^m$ is called \emph{predominantly asymptotically stable} if:
\begin{enumerate}
\item it is asymptotically stable relative to some $N \subset \RR^n$ and
\item 
$
\lim_{\varepsilon \rightarrow 0} \frac{\ell(B_\varepsilon(X)\cap N)}{\ell(B_\varepsilon(X))}=1.
$
\end{enumerate}
\end{definition}

Roughly speaking, a predominantly asymptotically stable set is an asymptotically stable set, up to a set with zero asymptotic Lebesgue measure (a wedge for instance).

\section{The main result: hypotheses and dynamical consequences}
\label{Section_main}
\subsection{The hypotheses}
Let $N \in \NN$ and $n\geq 3$. The object of our study is the dynamics around a heteroclinic network connecting equilibria, for which we give a rigorous description here. Specifically, we study a $C^2$--vector field (\ref{general}) on a compact set $G\subset\RR^n$, such that: 
\bigbreak
\begin{enumerate}
\item[\textbf{(H1)}]  its flow has $N$ hyperbolic equilibria, denoted by $p_1,p_2,\ldots,p_N$.
\bigbreak
\item[\textbf{(H2)}]  for each $i\in \{1,\ldots N\}$, the linearization of $f$ at $p_i$ has $n$ non-resonant\footnote{non-resonant in the sense of Hartman \cite{Hartman}.} eigenvalues denoted and ordered in the following way: 
\begin{equation}
\label{order}
\lambda_1^{(i)}>...>\lambda_{u_i}^{(i)}>0>-\lambda_{u_{i+1}}^{(i)}>...>-\lambda_{n}^{(i)} \quad \text{with} \quad \lambda_1^{(i)},\lambda_2^{(i)},..., -\lambda_{n-1}^{(i)}, -\lambda_{n}^{(i)} \in \RR^+.
\end{equation}
\bigbreak
\end{enumerate}
\medbreak
The positive and negative eigenvalues are called by expanding and contracting eigenvalues, respectively. For each $i\in \{1,\ldots N\}$, the eigendirections associated with the contracting eigenvalues at $p_i$ form the contracting tangent space at $p_i$.
\medbreak
The one-dimensional strongly stable manifold of $p_i$ is tangent to the eigendirection related to $-\lambda_{n}^{(i)}$ and, from now on, will be denoted by $W^{ss}(p_i)$. This manifold exists and is $C^2$-smooth -- details in \cite{HPS, Homburg96}. In the neighbourhood of $p_i$, there exists a local  $(n-1)$--dimensional $C^2$-smooth invariant manifold, denoted by $W^{cs}(p_i)$, tangent to the eigenspace associated with $$\lambda_2^{(i)}, \ldots, \lambda_{u_i}^{(i)},-\lambda_{u_i+1}^{(i)},\ldots,-\lambda_{n}^{(i)},$$ and often called by center stable manifold of $p_i$.  It is easy to see that $u_i$ is the Morse index of the saddle $p_i$. 
\bigbreak
For each $i \in \{1,  \ldots, N \}$, define the four positive constants:
\begin{equation}
\label{constants}
\alpha_i=\frac{\lambda_1^{(i)}}{\lambda_2^{(i)}}>1,  \qquad \beta_i=(\alpha_i)^{-1}=\frac{\lambda_2^{(i)}}{\lambda_1^{(i)}}<1, \qquad \mu_i=\frac{\lambda_{u_i+1}^{(i)}}{\lambda_1^{(i)}} \qquad \text{and} \qquad 
\rho_i= \frac{\lambda_1^{(i)}+\lambda_{
u_i+1}^{(i)}}{2\lambda_1^{(i)}}.
\end{equation}

\bigbreak
\begin{enumerate}
\item[\textbf{(H3)}] there exists a heteroclinic network $\Gamma$ in such a way that each equilibrium $p_i$ has a one-dimensional strong unstable manifold,  where the connection from $p_i$ to $p_{i+1}$ lies, with $i \in \{1,  \ldots, N \}$. Locally, each equilibrium has $u_i$ unstable and $s_i:=n-u_i$ stable directions. 
\end{enumerate}
\medbreak
The network $\Gamma$ may be a very intrincate heteroclinic network. Nevertheless, Hypothesis \textbf{(H3)} says that the equilibria are labelled in such a way that the strong unstable connection of $p_i$ joins $p_{i+1}$ where  $i \in \{1,  \ldots, N-1 \}$, which is always a closed path. With respect to the constant $\mu_i$, hereafter called the \emph{saddle-value} of $p_i$, we state our last hypothesis:
\medbreak
\begin{enumerate}
\item[\textbf{(H4)}] For each $i \in \{1, \ldots, N\}$, $\mu_i>1$.
\end{enumerate}
\bigbreak
\begin{remark}
Note that \textbf{(H1)}--\textbf{(H4)}  generalise the hypothesis of \cite[\S 5]{MTHAR}.
\end{remark}

\subsection{The principal cycle}
The heteroclinic cycle $\Gamma^p\subset \Gamma$ formed by the heteroclinic connections corresponding to the most positive expanding eigenvalue at each equilibrium, is called the \emph{principal cycle}. The letter $N^\star \leq N$ denotes the number of equilibria of the principal cycle. 
\medbreak

The set $$\Gamma^\star:=\bigcup _{i=1}^{N^\star} \{p_i\} \bigcup _{i=1}^{{N^\star}-1}[p_i \rightarrow p_{i+1}] \subset \Gamma $$ is called the \emph{principal heteroclinic sequence}, or heteroclinic channel according to \cite{MTHAR}. Note that the $\Gamma^\star\subset \Gamma^p\subset \Gamma$. 
\medbreak
Let $\varepsilon, \delta>0$ be arbitrarily small. 
Let $U_i$ be an open ball centered at $p_i$ and radius $\varepsilon>0$ that does not contain other equilibria but $p_i$. Since $\dim W^{cs}(p_i)= n-1$, near each $p_i$, the manifold $W^{cs}(p_i)$ locally divides $U_i$ into two open connected components, say $U_i^+$ and $U_i^-$.
For all $i \in \{1,\ldots, N-1\}$, we have $$[p_i \rightarrow p_{i+1}] \cap U_{i+1}^+ \neq \emptyset.$$  
Denote by 
$O_\delta([p_i \rightarrow p_{i+1}])$ a $\delta$-neighbourhood of $[p_i \rightarrow p_{i+1}]$ in $G\subset \RR^{n}$. Finally, define also  the following neighbourhood of the principal heteroclinic sequence $$V(\varepsilon, \delta)=\bigcup_{i=1}^{N^\star-1} O_\delta([p_i \rightarrow p_{i+1}])\bigcup_{i=1}^{{N^\star}-1} U_i.$$

\begin{figure}\begin{center}
\includegraphics[height=8cm]{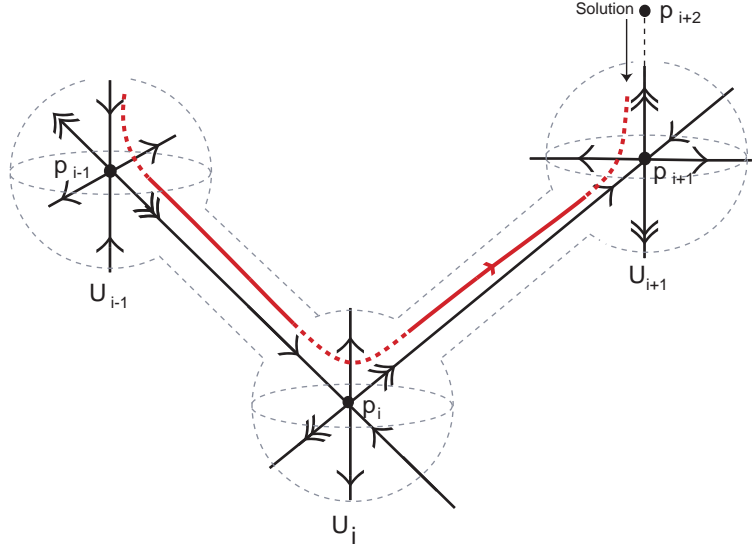}
\end{center}
\caption{\small Representation of a heteroclinic channel and a solution following it. Double arrows correspond to the strong expanding directions }
\label{network1}
\end{figure}

\begin{definition}
\label{channel}
 Let $\varepsilon, \delta>0$ sufficiently small.
We say that system (\ref{general}) has a heteroclinic channel inside $V(\varepsilon, \delta)$ if there exist $T\geq 0$ and an open set $U\subset U_1^+$ of initial conditions such that every solution $\varphi(t,x)$, $x\in U$ satisfy (see Figure \ref{network1}):
\begin{enumerate}
\item $\varphi(0,x)=x$;
\item $\forall t \in [0,T]$, $\varphi(t,x) \in V(\varepsilon, \delta)$;
\item $\forall i \in \{1, \ldots, N^\star\}$, there exists a finite sequence $(t_i)_i$, $t_i<T$, such that $\varphi(t_i,x)\in U_i$.
\end{enumerate}
\end{definition}
\bigbreak
The next result says that the occurrence of a heteroclinic channel inside $V(\varepsilon, \delta)$ is a robust property (provided $\varepsilon>0$ and $\delta>0$ are sufficiently small). 

\begin{theorem}[\cite{ATHR, MTHAR}, adapted]
\label{Main1}
Let $\varepsilon, \delta>0$ sufficiently small. 
If  (\ref{general}) satisfies the Hypotheses \textbf{(H1)}--\textbf{(H4)}, then:
\begin{enumerate}
\item there exists a heteroclinic channel inside $V(\varepsilon, \delta)$ and 
\item there exists a neighbourhood $\mathcal{V}$ of the vector field $f \in \mathcal{B}$, endowed with the $C^1$ norm, and an open set $\mathcal{U}\subset \mathcal{V}$ such that the flow of $\dot{x}=h(x)$ has a heteroclinic channel inside $V(\varepsilon, \delta)$, for every $h \in \mathcal{U}$.
\end{enumerate}
 \end{theorem}

Althought the existence of a heteroclinic channel is a robust property, the channel as in Definition \ref{channel}, might be a transient phenomenon. 

\subsection{The main result}

In the present paper, we prove a stronger version of Theorem \ref{Main1}.  
We show that if a heteroclinic network $\Gamma$ is asymptotically stable then all connections corresponding to the most positive expanding eigenvalues
of the linearization at the equilibria will generically form a part of an attractor in the sense of \cite{Melbourne}. Although the existence of the heteroclinic cycle is not $C^1$-structurally stable, the attracting channel is robust. Numerical results described in \cite{MTHAR} show that the action of small noise does not destroy the channel.  The term ``generic'' corresponds to an open condition that be specified later and has the same flavour as \cite[Lemma 2.3(h)]{Melbourne}.  
\begin{theorem}
\label{Main}
If  (\ref{general}) satisfies the Hypotheses \textbf{(H1)}--\textbf{(H4)} then, generically, the cycle $\Gamma^p\subset \Gamma$ corresponding to the most positive expanding eigenvalues of the linearization at the equilibria is predominantly asymptotically stable.
\end{theorem}

In other words, Theorem \ref{Main} says that, generically, the strong unstable connections define a predominantly asymptotically stable Milnor attractor within $\Gamma$ that is not Lyapunov stable. Along the proof, the reader will realize that, given a small neighbourhood of $\Gamma^p$, there exists a cuspoidal region with asymptotically zero Lebesgue measure  such that trajectories are repelled from it as $t \rightarrow +\infty$. 
 
\subsection{Important remarks}
In what follows, we discuss the relevance of the hypotheses and we also refer some differences between our main result and others in the literature.
\medbreak
\begin{enumerate}
\item Our result is consistent with the numerics of \cite{ATHR, MTHAR} developed in the context of the Generalized Lotka-Volterra model. See also \cite{Bick}. Our result is completely compatible to that of \cite[Prop. 7.6 and 7.8]{Podvigina} in which the authors proved that finding trajectories that follow the strong unstable connection within a stable network is a lot more probable than finding those that follow the other connections. 
\medbreak
\item Throughout our proof, we assume the absence of transverse eigenvalues to the principal cycle $\Gamma^p$ -- see Remark \ref{Rem1}.
Moreover, we do not use the fact that the vector field $f$ commutes with the action of a compact Lie group. Typically, the fixed-point subspaces forces the generic assumption of Theorem \ref{Main} to fail.
\medbreak
 \item Theorem \ref{Main} is remarkably different to that of \cite{LR, Rodrigues3} in which the leading eigenvalues at the equilibria are non-real. Due to the transversality of the invariant manifolds combined with complex eigenvalues, the horseshoe dynamics does not trap most solutions in the neighbourhood of the cycle. In this case, nearby trajectories seem to be equally distributed between the different connections.
 \medbreak
 \end{enumerate}
\subsection{Kirk and Silber example}
The authors of \cite{KS}  discussed the case of two competing cycles in a system of ordinary differential equations in $\RR^4$ with $\ZZ^4_2$--symmetry. Their system has four equilibria $\xi_j$ on the coordinate axes at $x_j=1$, $j=1,2,3,4$, and two competing cycles (see Figure \ref{KS}):
$$
\xi_1 \rightarrow \xi_2 \rightarrow \xi_3  \rightarrow \xi_1 \qquad \text{and}  \qquad \xi_1 \rightarrow \xi_2 \rightarrow \xi_4  \rightarrow \xi_1.
$$
Note that there is an extra saddle equilibrium point whose unstable manifold is two-dimensional, which we call $\eta$.  The one-dimensional cycles share the connection $[\xi_1 \rightarrow \xi_2 ]$ and the unstable manifold of $\xi_2$ is two-dimensional. Moreover, there are two positive coefficients in the differential equation, $e_{23}>0$ and $e_{24}>0$, that correspond to the expanding eigenvalues of $\xi_2$. Numerics allow to prove that, when $e_{23}>e_{24}$, solutions leaving $\xi_2$ in the  direction of $\xi_4$ pass through a cuspoidal region abutting the connection $[\xi_1 \rightarrow \xi_2]$, where both cycles may be attracting. This example does not fit in our study because the connections from $\xi_2$ to $\xi_3$ involve a one-dimensional and a two-dimensional connecting sets.

\begin{figure}[ht]
\begin{center}
\includegraphics[height=5cm]{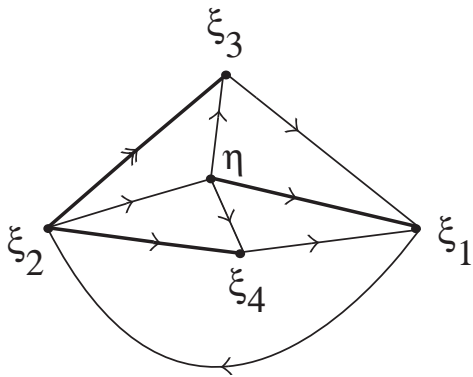}
\end{center}
\caption{\small Representation of the simplest configuration for the network considered in \cite{KS} (simplified version). The bold  lines represent two-dimensional connections; thin lines correspond to one-dimensional connections. Double arrows correspond to the strong expanding directions (case $e_{23}>e_{24}$); $\eta$ corresponds to the extra saddle equilibrium. }
\label{KS}
\end{figure}

\section{Linearization and Return Maps}
\label{Linearization}
In order to study homo and heteroclinic bifurcations, two approaches have been taken. In the first one, due to Shilnikov and rigorously proved in Deng \cite{Deng1}, one rewrites the differential equation in its integral form and uses smoothness results for the integral equations to derive approximations of the Poincar\'e map. A different technique, used in \cite{Tresser} and coworkers, uses linearization results obtained via a normal form procedure. In this paper, we are going to use the second approach; we establish local coordinates near the hyperbolic equilibria and define some terminology that will be used in the rest of the paper.
\medbreak

Let $i \in \{1, \ldots, N\}$. Since $p_i$ is a hyperbolic equilibrium of (\ref{general}), by the invariant manifold theory \cite{HPS}, there exists a $C^2$ local coordinates $(x,y)\in \RR^{n}$ in the neighbourhood of $p_i$ so that the local stable manifold of $p_i$ is $x=0$ and the local unstable manifold is $y=0$. This means that we can essentially reduce the problem by looking at some small neighbourhood of $p_i$.
In this case, the linearization of the vector field $f$ at $p_i$ may be represented by a diagonal matrix after a $C^1$ change of coordinates. Assume that we may decompose the tangent space into:
\begin{equation}
\label{splitting}
T_{p_i} \RR^{n}= T{c_i}\oplus T{e_i},
\end{equation}
where $T{c_i}$ and  $T{e_i}$ are the contracting and the expanding eigenspace at $p_i$, respectively (details in \cite{KM1}).
\begin{remark}
\label{Rem1}
With the direct sum (\ref{splitting}), we are using explicitly  that there are no transverse nor radial eigenvalues. Nevertheless our main result is still valid if all eigenvalues in these directions are less than zero.
\end{remark}

By \textbf{(H2)}, the expanding and contracting tangent spaces $Te_i$ and $Tc_i$ are assumed to be $u_i$ and $s_i$--dimensional. Sometimes it will be useful to write $x\in \RR^{u_i}$ and $y\in \RR^{s_i}$ as $(x_1,\ldots, x_{u_i})\in \RR^{s_i}$ and $(y_1,\ldots, y_{s_i})\in \RR^{s_i}$. In a neighbourhood of $p_i$, hereafter called $U_i$, we transform the vector field into the linearized form:
\begin{equation}
\label{linearization1}
\left\{ 
\begin{array}{l}
\dot{x}^{(i)}=A^{(i)} x^{(i)} \\ \\

\dot{y}^{(i)}=B^{(i)} y^{(i)} 
\\
\end{array}
\right.
\end{equation}
where $A$ and $B$ have all positive and negative eigenvalues, respectively. The conditions for $C^1$--linearization of Hartman  \cite{Hartman} of \textbf{(H2)} show that linearization is not possible for subsets of points on the lines defined by resonances  (\emph{i.e.} is a generic condition). 
The restrictions are a
set of zero Lebesgue measure in the parameter space and place no serious constraint on the analysis that follows.

\begin{figure}\begin{center}
\includegraphics[height=6cm]{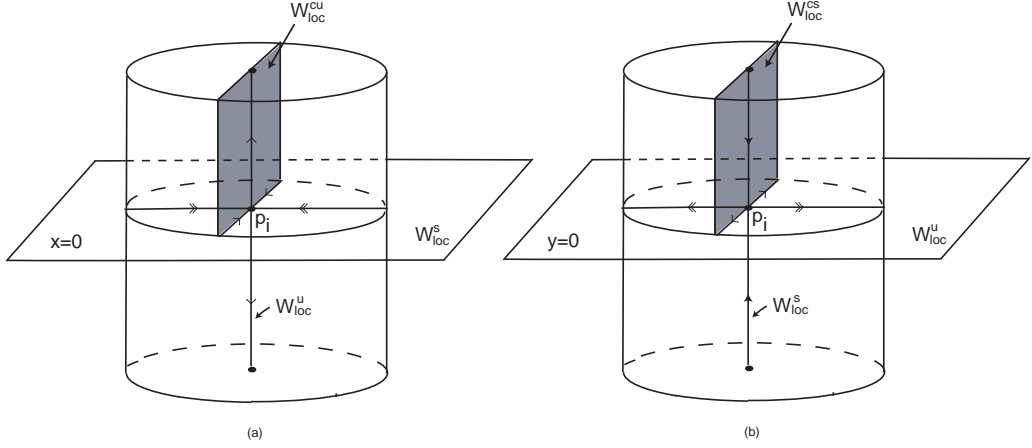}
\end{center}
\caption{\small Neighbourhood of $p_i$. Representation of the cross sections (a): $\Sigma^{in}_i$ and (b): $\Sigma^{out}_i$. Trajectories starting at interior points of $\Sigma^{in}_i$ go inside $U_i$ in positive time. Trajectories starting at interior points of $\Sigma^{out}_i$ go outside $U_i$ in positive time. }
\label{neigh1}
\end{figure}

\subsection{Local coordinates and cross sections}
Let $i\in \{1,\ldots, N\}$. By scaling the local coordinates of the neighbourhood $U_i$ of the equilibrium $p_i$, put $A^{(i)}$ and $B^{(i)}$ into Jordan normal form and write:
$$
A^{(i)}= \begin{pmatrix} \lambda_1^{(i)} & 0 &... &0 \\ 0 & \lambda_2^{(i)} & ... &0 \\ 0 &0&... & 0 \\ 0 &0 &... &\lambda_{u_i}^{(i)}   \end{pmatrix}
\qquad \text{and}
\qquad 
B^{(i)}= \begin{pmatrix} -\lambda_{u_i+1}^{(i)} & 0 &... &0 \\ 0 & -\lambda_{u_i+2}^{(i)} & ... &0 \\ 0 &0&... & 0 \\ 0 &0 &... &-\lambda_{n}^{(i)}   \end{pmatrix}.
$$
Define the cross section $\Sigma^{in}_i\subset \RR^{n}$ as :
$$
\Sigma^{in}_i=\left\{\left(x^{(i)},y^{(i)}\right) : \left\|y^{(i)}\right\|=1\right\}=\left\{\left(x^{(i)},\theta^{(i)}\right), x^{(i)} \in \RR^{u_i}, \theta^{(i)} \in \EU^{s_i-1}\right\}.
$$
This set intersects all trajectories approaching some neighbourhood of $p_i$. Analogously, we may define another cross section:
$$
\Sigma^{out}_i=\left\{\left(x^{(i)},y^{(i)}\right) : \left\|x^{(i)}\right\|=1\right\}=\left\{\left(\phi^{(i)},y^{(i)}\right), \phi^{(i)} \in \EU^{u_i-1}, y^{(i)} \in \RR^{s_i} \right\}.
$$

These cross sections, depicted in Figure \ref{neigh1}, are transverse to the flow. By construction, trajectories starting at interior points of $\Sigma^{in}_i$ go inside $U_i$ in positive time. Trajectories starting at interior points of $\Sigma^{out}_i$ go outside $U_i$ in positive time.  The sets $\EU^{s_i-1}$ and $\EU^{u_i-1}$ are obtained by identifying the opposite faces of $\Sigma_i^{in}$ and $\Sigma_i^{out}$.  Intersections between the local invariant manifolds at $p_i$ and cross sections are parametrized by:
$$ 
W^s_{loc}(p_i)\cap \Sigma^{in}_i=\left\{x^{(i)}=0_{\RR^{u_i}}\right\} 
\qquad \text{and} \qquad 
W^u_{loc}(p_i)\cap \Sigma^{out}_i=\left\{y^{(i)}=0_{\RR^{s_i}}\right\}$$

\subsection{Time of flight}
Let $i\in \{1,\ldots, N\}$. Hereafter, when we refer $\RR^{u_i}\subset \Sigma_i^{in}$ we mean $\RR^{u_i}\times \{0_{\EU^{s_i-1}}\}\subset \Sigma_i^{in}$. The same for $\Sigma_i^{out}$.
Define implicitly the map $$T_i: \RR^{u_i}\backslash\{0_{\RR^{u_i}}\}\subset \Sigma_i^{in} \rightarrow \RR^+_0$$ as $\left\|x^{(i)}\right\|=1.$ In other terms, the map $T_i$ is defined in such a way that:
$$
 \sum_{j=1}^{u_i} \left(x_j^{(i)}\right)^2 \exp\left({2\lambda_j^{(i)} T_i\left(x^{(i)}\right)}\right)=1.
$$
The map $T_i$ gives the transit time that the solution with initial condition $\left(x^{(i)},y^{(i)}\right)$ spends inside $U_i$.
Define also  the map $\tau^{(i)}: \RR^{u_i}\backslash\{0_{\RR^{u_i}}\} \rightarrow \EU^{u_i-1} $
by
$$
\tau^{(i)}\left(x_1^{(i)}, \ldots, x_{u_i}^{(i)}\right) = \left(\tau_1^{(i)}\left(x^{(i)}\right), \ldots, \tau_{u_i}^{(i)}\left(x^{(i)}\right)\right)
\quad
\text{where} 
\quad 
\tau_j^{(i)}\left(x^{(i)}\right)=\exp\left({\lambda_j T_i\left(x^{(i)}\right)}\right) x_j^{(i)},$$
for $ j=1,\ldots, {u_i}.$
For all $i\in\{1,\ldots, N\}$ and $j\in\{1,\ldots, n\}$, we know that $\tau^{(i)} \in \EU^{{u_i}-1}$. Thus $\left|\tau_j^{(i)}\right|\leq 1$ for all $j$. In particular, for $j=1$, we get:
$$
\left|\tau_1^{(i)}\left(x^{(i)}\right)\right|=\left|\exp\left(\lambda_1 T_i\left(x^{(i)}\right)\right) x_1^{(i)}\right|=\left|\exp\left(\lambda_1 T_i\left(x^{(i)}\right)\right)\right|\left| x_1^{(i)}\right|\leq 1.
$$
It follows straightforwardly that:
\begin{equation}
\label{important to lemma 4}
\left| x_1^{(i)}\right|^{-\frac{1}{\lambda_1^{(i)}}}\geq \exp \left(T_i\left(x^{(i)}\right)\right).
\end{equation}

\bigbreak
 The following three technical results will be useful in the sequel.
 
\begin{lemma}
\label{lemma3}
Let $i\in \{1,\ldots, N\}$. The following two inequalities hold:
$$
\ln \left\|x^{(i)}\right\|^{-\frac{1}{\lambda_1^{(i)}}} < T_i\left(x^{(i)}\right) < \ln \left\|x^{(i)}\right\|^{-\frac{1}{\lambda_{u_i}^{(i)}}}.
$$
\end{lemma}

\begin{proof}
By \textbf{(H2)}, since $\lambda_1^{(i)}>...>\lambda_{{u_i}}^{(i)}>0$, the following implications hold:

\begin{align*}
\sum_{j=1}^{u_i} \left(x_j^{(i)}\right)^2 \exp\left({2 \lambda_j^{(i)} T_i(x^{(i)})}\right)=1 \quad & \Rightarrow  \sum_{j=1}^{u_i} \left(x_j^{(i)}\right)^2 \exp\left({2 \lambda_1^{(i)} T_i(x^{(i)})}\right)>1 \\
& \Leftrightarrow \exp\left({2 \lambda_1^{(i)}} T_i\left(x^{(i)}\right)\right)  \sum_{j=1}^{u_i} \left(x_j^{(i)}\right)^2 >1 \\
& \Leftrightarrow  \exp\left({2 \lambda_1^{(i)} T_i\left(x^{(i)}\right)}\right)>\frac{1}{\left\|x^{(i)}\right\|^2} \\
& \Leftrightarrow  \lambda_1^{(i)} T_i\left(x^{(i)}\right)>-\ln\left({\left\|x^{(i)}\right\|}\right) \\
& \Leftrightarrow  \ln \left\|x^{(i)}\right\|^{-\frac{1}{\lambda_1^{(i)}}} < T_i\left(x^{(i)}\right).
\end{align*}

The first inequality of the lemma is shown; the  proof of the other is analogous.
\end{proof}
In what follows, recall that if $i\in \{1,\ldots, N\}$ then $\beta_i=(\alpha_i)^{-1}=\frac{\lambda_2^{(i)}}{\lambda_1^{(i)}}<1$.

\begin{lemma} 
\label{lemma4}
Let $i\in \{1,\ldots, N\}$. The following inequality holds: 
$$
1-\left(\tau_1^{(i)}\left(x^{(i)}\right)\right)^2< \left|x_1^{(i)}\right|^{-2\beta_i} \sum_{j=2}^{u_i}  \left(x_j^{(i)} \right)^2
$$
\end{lemma}

\begin{proof}
The proof runs along the same lines to that of Lemma \ref{lemma3}. By \textbf{(H2)}, since $\lambda_2^{(i)}>...>\lambda_{{u_i}}^{(i)}>0$, the following equalities and inequalities are valid:
\begin{align*}
   1-\left(\tau_1^{(i)}\left(x^{(i)}\right)\right)^2 & = \sum_{j=2}^{u_i} \left(\tau_j^{(i)}\left(x^{(i)}\right)\right)^2 \\
    & = \sum_{j=2}^{u_i} \exp\left(2\lambda_j^{(i)} T_i\left(x^{(i)}\right)\right) \left(x_j^{(i)}\right)^2   \\
    & = \sum_{j=2}^{u_i} \exp\left(T_i\left(x^{(i)}\right)\right)^{2\lambda_j^{(i)}} \left(x_j^{(i)}\right)^2 \\
    & < \sum_{j=2}^{u_i} \exp\left(T_i\left(x^{(i)}\right)\right)^{2\lambda_2^{(i)}} \left(x_j^{(i)}\right)^2 \\
    & = \exp\left(T_i\left(x^{(i)}\right)\right)^{2\lambda_2^{(i)}}  \sum_{j=2}^{u_i}  \left(x_j^{(i)}\right)^2 \\
     & \leq\left|x_1^{(i)}\right|^{-2\beta_i}  \sum_{j=2}^{n-s_i}  \left(x_j^{(i)}\right)^2.
     \end{align*}
     The last inequality follows from inequality (\ref{important to lemma 4}). 
\end{proof}

\begin{lemma}
\label{lemma com k} Let $i\in \{1,\ldots, N\}$. Let $k \in \RR^+\backslash\{0\}$ be such that  $\left(x_1^{(i)}\right)^2>k \sum_{j=2}^{u_i}\left(x_j^{(i)}\right)^2.$ Then:
$$
1-\left(\tau_1\left(x^{(i)}\right)\right)^2<k ^{-\beta_i} \left\| x^{(i)}\right\|^{2-2\beta_i}.
$$
\end{lemma}

\begin{proof}
Let us fix a $k \in  \RR^+\backslash \{0\}$ such that $\left(x_1^{(i)}\right)^2>k \sum_{j=2}^{u_i}\left(x_j^{(i)}\right)^2$. Thus, the following equalities and inequalities hold:
\begin{align*}
   1-\left(\tau_1^{( i)}\left(x^{( i)}\right)\right)^2 & = \sum_{j=2}^{u_i} \left(\tau_j^{( i)}\left(x^{( i)}\right)\right)^2 \\
     & = \sum_{j=2}^{u_i} \exp\left(2\lambda_j^{( i)} T_{ i}\left(x^{( i)}\right)\right) \left(x_j^{( i)}\right)^2   \\
     & < \sum_{j=2}^{u_i} \exp\left(T_{ i}\left(x^{( i)}\right)\right)^{2\lambda_2^{( i)}} \left(x_j^{( i)}\right)^2 \\ 
    & = \exp\left(T_i\left(x^{(i)}\right)\right)^{2\lambda_2^{(i)}}  \sum_{j=2}^{u_i}  \left(x_j^{(i)}\right)^2  \\
     & \leq \left|x_1^{(i)}\right|^{-2\beta_i}  \sum_{j=2}^{u_i}  \left(x_j^{(i)}\right)^2 \qquad (\text{By Remark } \ref{important to lemma 4})\\
     & <  k ^{-\beta_i} \left\| x^{( i)}\right\|^{2-2{\beta_i}}\\
      \end{align*}
      The hypothesis is used on the last inequality. 
\end{proof}

\begin{figure}\begin{center}
\includegraphics[height=5cm]{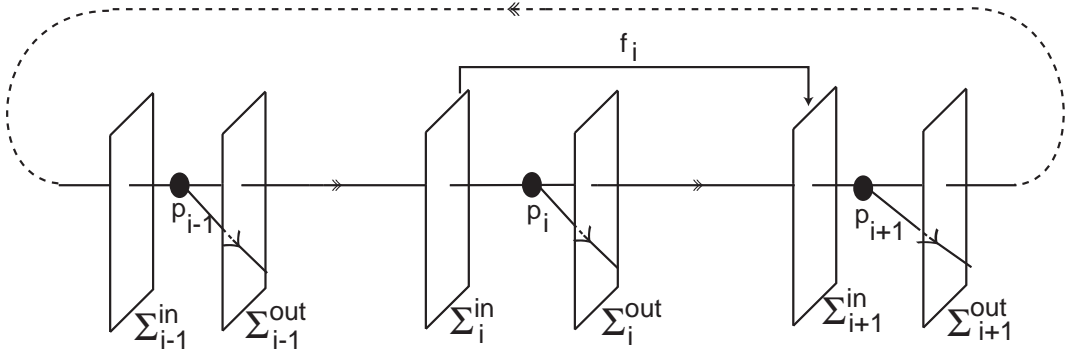}
\end{center}
\caption{\small Representation of the maps  $f_i: \Sigma_i^{in} \rightarrow \Sigma_{i+1}^{in}$ and $F: \Sigma_1^{in} \rightarrow \Sigma_{1}^{in}$.  }
\label{local_map}
\end{figure}

\subsection{The return map}
As in \cite{AC}, we are interested in a transition model, that we denote by $f_i$, that carry points from $\Sigma_i^{in}$ to $\Sigma_{i+1}^{in}$. As depicted in Figure \ref{local_map}, the map $f_i: \Sigma_i^{in} \rightarrow \Sigma_{i+1}^{in}$ may be written as:
$$
f_i\left(x^{(i)},y^{(i)}\right)=\left(x^{(i+1)}, y^{(i+1)}\right)
$$
where:
  \begin{equation}
\label{linearization2}
\left\{ 
\begin{array}{l}
x^{(i+1)}= M^{(i)}\left(\tau^{(i)} \left(x^{(i)}\right)\right)\left[\exp\left({-\lambda_{u_i+1}^{(i)} T_i\left(x^{(i)}\right)}\right)y_1^{(i)}, ..., \exp\left({-\lambda_{n}^{(i)} T_i\left(x^{(i)}\right)}\right)y_{s_i}^{(i)}\right]\in \RR^{u_{i+1}}\bigbreak \\ 

y^{(i+1)}=\gamma^{(i)}\left(\tau^{(i)}\left(x^{(i)}\right)\right)\in \EU^{s_{i+1}-1}
\end{array}
\right.
\end{equation}
where $N^\star +1 \equiv 1$ and  $M^{(i)}: \EU^{u_i-1} \rightarrow GL(\RR^{u_{i+1}})$ and $\gamma^{(i)}:\EU^{u_i-1} \rightarrow\EU^{s_{i+1}-1}$ are smooth maps. Here, by smooth, we mean at least $C^1$. 
\bigbreak

Now let  $F: \Sigma_1^{in} \rightarrow \Sigma_{1}^{in}$ be $F= f_{N^\star} \circ f_{N^\star-1} \circ \ldots \circ f_1$ defined as:
$$
F\left(x^{(1)},y^{(1)}\right)=\left(x^{(1)}_\star, y^{(1)}_\star\right)
$$
with:
  \begin{equation}
\label{linearization3}
\left\{ 
\begin{array}{l}
x^{(1)}_\star= M\left(\tau^{(1)} \left(x^{(1)}\right)\right) \left[\exp\left({-\lambda_{u_i+1}^{(N^\star)} T_{N^\star}\left(x^{(N^\star)}\right)}\right)y_1^{(N^\star)}, ..., \exp\left({-\lambda_{n}^{(N^\star)} T_{N^\star}\left(x^{(N^\star)}\right)}\right)y_{s_i}^{(N^\star)}\right] \bigbreak \\
 y^{(1)}_\star=\gamma\left(\tau^{(1)}\left(x^{(1)}\right)\right)
\end{array}
\right.
\end{equation}
where:
$$
\left\|M\right\|\leq \prod_{i=1}^{N^\star} \left\|M^{(i)}\right\| \qquad \text{and} \qquad \gamma= \gamma^{(N^\star)} \circ \ldots \circ \gamma^{(1)}.
$$

\section{Proof of the Main Result}
\label{Proof of the Main Theorem}

The main goal of this section is the proof of Theorem \ref{Main}.
First of all note that,  under conditions \textbf{(H1)--(H4)}, every trajectory starting in any small neighbourhood of $p_1$ remains in a neighbourhood of $\Gamma$ until it comes to a neighbourhood of $p_N$ -- details in \cite{AZR}. In the equivariant context, the stability of $\Gamma$ can be obtained using \cite{MTHAR} and \cite{KM1, KM2}. 

 The rest of the proof will be based in \cite{ATHR, AC, Brannath, Deng1, KS, Melbourne}.  We use the previous section to show that for asymptotically stable networks, under $N^\star$ open conditions on the space of parameters, the majority of trajectories follow a cycle formed by heteroclinic trajectories along the strongly unstable directions.
 
 \medbreak
 The proof will be divided into several lemmas.  
 If $X \subset \RR^n$, let ${X}^c$ be the set $\RR^n \backslash X$.

\subsection{A repelling cuspoidal region}
Let $i\in\{1, \ldots, N\}$.
As suggested by Figure \ref{transition_sphere} for $n=2$, define:
\begin{equation}
\label{arrows}
\left\{ 
\begin{array}{l}
e_{\pm}^{(i)} =(\pm 1, 0,\ldots, 0) \in \EU^{s_i-1} \subset \Sigma^{out}_i \bigbreak \\ 
y_{\pm}^{(i+1)} =\gamma^{(i)}\left(e_{\pm}^{(i)}\right) \in \EU^{s_{i+1}-1} \subset \Sigma^{in}_{i+1} \bigbreak  \\
b_{\pm}^{(i)} =(0, 0,\ldots, \pm 1) \in \EU^{s_{i+1}-1} \subset \Sigma^{in}_i
\end{array}
\right.
\end{equation}

\begin{figure}\begin{center}
\includegraphics[height=5cm]{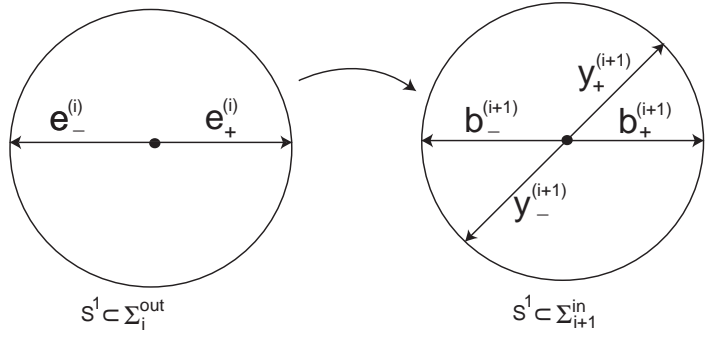}
\end{center}
\caption{\small Representation of $e_{\pm}^{(i)}\in \Sigma^{out}_i$, $y_{\pm}^{(i+1)}=\gamma^{(i)}\left(e_{\pm}^{(i)}\right) \in\Sigma^{in}_{i+1}$ and  $b_{\pm}^{(i)}\in\Sigma^{in}_{i+1}$, for $n=2$. }
\label{transition_sphere}
\end{figure}

Note that, by construction, the set $\Gamma^{p}$ corresponds to the set of connections that cross $\Sigma_i^{in}$ at $y^{(i)}=y^{(i)}_\pm$ and then $\Sigma_{i}^{out}$ at $x^{(i)}=e_\pm^{(i)}$. If $\delta>0$ is sufficiently small, let us define a $\delta$-neighbourhood of the heteroclinic connection $[p_{i-1} \rightarrow p_{i}]$ within $\Sigma_i^{in}$ as:
$$
E_i(\delta)=\left\{\left(x^{(i)},y^{(i)}\right) \in \Sigma_i^{in} : \left\|x^{(i)}\right\|<\delta\right\}
$$
and the two following sets:
\begin{equation}
\label{sets1}
\left\{ 
\begin{array}{l}
\mathcal{F}_i(\delta)= \left\{\left(x^{(i)},y^{(i)}\right) \in \Sigma_i^{in}: \min\left\{\left\|y^{(i)}-y_+^{(i)}\right\|, \left\|y^{(i)}-y_-^{(i)}\right\|\right\}<\delta\right\}\\
\\
B_i(\delta)=E_i(\delta)\cap \mathcal{F}_i(\delta)
\end{array}
\right.
\end{equation}
\bigbreak
Let $i\in\{1, \ldots, N\}$ and $\varepsilon>0$ sufficiently small. 

\begin{definition}
A $\varepsilon$-\emph{wedge} in $\Sigma_i^{in}$ is defined by:
\begin{align*}
 \mathcal{W}_i(\varepsilon)= & \left\{\left(x^{(i)},y^{(i)}\right)\in \Sigma_i^{in}: 1- \left(\tau_1^{(i)}\left(x^{(i)}\right)\right)^2<\varepsilon^2\right\} \\
 = & \left\{\left(x^{(i)},y^{(i)}\right)\in \Sigma_i^{in}: 1- \exp\left(2\lambda_1^{(i)}T_i\left(x^{(i)}\right)\right)x_1^{(i)}<\varepsilon^2\right\}. 
\end{align*}
\end{definition}

\begin{lemma}
\label{facil}
The $\varepsilon$-\emph{wedge} $\mathcal{W}_i(\varepsilon)$  contains all initial conditions in $\Sigma_i^{in}$ that are mapped to within $\varepsilon$ of $W^{uu}_{loc}(p_i)$ by the projection of the local map (near $p_i$) induced by the linearization onto $\EU^{s_i-1}\subset \Sigma_i^{out}$. 
\end{lemma}

\begin{proof}
The projection of the local map  into  $\EU^{s_i-1}\subset \Sigma_i^{out}$ depends on $\left(x_1^{(i)}, \ldots, x_{s_i}^{(i)}\right)$ and is given by $\tau^{(i)}\left(x^{(i)}\right)$. The points in $\Sigma_i^{in}$ that are mapped to within $\varepsilon$ of $W^{uu}_{loc}(p_i)$ corresponds to that:
$$
\left|\tau_1^{(i)}\left(x^{(i)}\right)-1\right|< \varepsilon \qquad \text{and} \qquad \left|\tau_1^{(i)}\left(x^{(i)}\right)+1\right|< \varepsilon,
$$
implying that $$\left|\tau_1^{(i)}\left(x^{(i)}\right)-1\right|\left|\tau_1^{(i)}\left(x^{(i)}\right)+1\right|= 1- \left(\tau_1^{(i)}\left(x^{(i)}\right)\right)^2<\varepsilon^2. 
$$
See Figure \ref{cusp1}. 
\end{proof}

Next result says that  $ \mathcal{W}_i(\varepsilon)$ has a cuspoidal form, implying that $\mathcal{W}_i(\varepsilon)^c$ has asymptotically full measure near the set of connections $\EU^{s_i-1}\subset \Sigma^{in}_i$. It may be seen as a consequence of Deng's Strong Lambda Lemma \cite{Deng1}.

\begin{proposition}
\label{medida da cusp}
For each $i\in\{1, \ldots, N\}$, given $\varepsilon, \delta>0$, we have:
$$
\frac{\ell (\mathcal{W}_i(\varepsilon)^c \cap E_i(\delta))}{\ell(E_i(\delta))} \leq \varepsilon^{-2\alpha_i}\delta^{\alpha_i-1}.$$ 
\end{proposition}

\begin{figure}\begin{center}
\includegraphics[height=6cm]{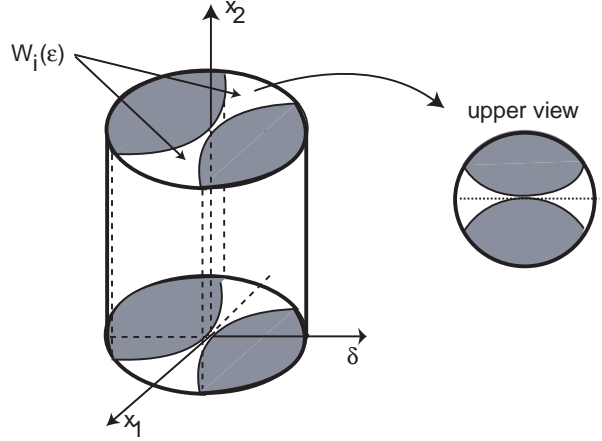}
\end{center}
\caption{\small Representation of the cusp $\mathcal{W}_i(\varepsilon)\subset \Sigma_i^{in}$, where $i\in \{1, \ldots, N\}$. The cuspoidal region corresponds to the white region. }
\label{cusp1}
\end{figure}

\begin{proof}
By Lemma \ref{lemma4}, it is easy to see that:
\begin{align*}
   \mathcal{W}_i(\varepsilon)^c \cap E(\delta)= 
    &\left\{\left(x^{(i)},y^{(i)}\right) \in \Sigma_i^{in} : \left\|x^{(i)}\right\|<\delta \quad \text{and} \quad 1- \left(\tau_1^{(i)}\left(x^{(i)}\right)\right)^2\geq \varepsilon^2\right\} \\    
   \subset & \left\{\left(x^{(i)},y^{(i)}\right) \in \Sigma_i^{in} : \left\|x^{(i)}\right\|<\delta \quad \text{and} \quad  \varepsilon^2\left(x_1^{(i)}\right)^{2\beta_i}\leq \sum_{j=2}^{n} \left(x_j^{(i)}\right)^2\right\} \\
  \subset & \left\{\left(x^{(i)},y^{(i)}\right)\in \Sigma_i^{in} : \left\|x^{(i)}\right\|<\delta \quad \text{and} \quad  \left(x_1^{(i)}\right)^{2}\leq \left(\frac{\delta^2}{\varepsilon^2}\right)^{\alpha_i}\right\} .
 \end{align*}
  The last inclusion follows because:
 $$
 \varepsilon^2\left(x_1^{(i)}\right)^{2\beta_i}\leq \sum_{j=2}^{n-s_i} \left(x_j^{(i)}\right)^2\leq \sum_{j=1}^{n-s_i} \left(x_j^{(i)}\right)^2=\left\|x^{(i)}\right\|^2<\delta^2.
 $$
 Since $\beta_i=(\alpha_i)^{-1}$, then:
 $$
 \varepsilon^2\left(x_1^{(i)}\right)^{2\beta_i}<\delta^2 \quad \Rightarrow\quad  \left(x_1^{(i)}\right)^{2}\leq \left(\frac{\delta^2}{\varepsilon^2}\right)^{\alpha_i}.
 $$ 
Integrating and using Fubini's Theorem, it follows immediately that: $$\frac{\ell (\mathcal{W}_i(\varepsilon)^c \cap E_i(\delta))}{\ell(E_i(\delta))} \leq {\varepsilon^{-2\alpha_i}\delta^{\alpha_i-1}}.$$
\end{proof}
By \textbf{(H4)}, for all $i\in\{1, \ldots, N\}$, we have $\alpha_i>1$. Hence:

\begin{corollary}
\label{medida da cusp}
For each $i\in\{1, \ldots, N\}$, if $\varepsilon, \delta>0$ are sufficiently small, the following equality holds:
$$
\lim_{\delta \rightarrow 0}\frac{\ell (\mathcal{W}_i(\varepsilon)^c \cap E_i(\delta))}{\ell(E_i(\delta))} =0.
$$
\end{corollary}

\bigbreak
\subsection{Uniform convergence}
The next result says that the image of the wedge $\mathcal{W}_i(\varepsilon)$ under $f_i$ lies in a neighbourhood of $y^{(i+1)}_\pm$. 

\begin{lemma}
\label{Lemma10}
For each $i\in\{1, \ldots, N\}$, there is $\varepsilon_1^{(i)}>0$  such that for all $\varepsilon\in \left[0,\varepsilon_1^{(i)}\right]$, we have: $$f_i(\mathcal{W}_i(\varepsilon))\subset \mathcal{F}_{i+1}\left(\chi_i \varepsilon\right) \qquad \text{where} \qquad \chi_i=\max_{\phi^{(i)}\in\, \EU^{s_i-1}} \left\|\frac{d\gamma^{(i)}}{d\phi^{(i)}}\right\|.$$
\end{lemma}
\bigbreak
\begin{proof}
We want to prove that if $\left(x^{(i+1)}, y^{(i+1)}\right) \in f_i(\mathcal{W}_i(\varepsilon))$, then $\left(x^{(i+1)}, y^{(i+1)}\right) \in \mathcal{F}_{i+1}(\chi_i\varepsilon)$. Indeed, if $\left(x^{(i+1)}, y^{(i+1)}\right) \in f_i(\mathcal{W}_i(\varepsilon))$, then there exists $\left(x^{(i)}, y^{(i)}\right) \in \mathcal{W}_i(\varepsilon)$ such that 
$$
f_i\left(x^{(i)}, y^{(i)}\right) =\left(x^{(i+1)}, y^{(i+1)}\right).
$$ 
Since $\left(x^{(i)}, y^{(i)}\right) \in \mathcal{W}_i(\varepsilon)$ then, by Lemma \ref{facil}, we have $\left\| y^{(i)}-e_\pm^{(i)}\right\|<\varepsilon$. 
Therefore, 
\begin{align*}
\left\|y^{(i+1)}-\gamma^{(i)}\left(e_\pm^{(i)}\right)\right\| & = \left\|\gamma^{(i)}\left(y^{(i)}\right)-\gamma^{(i)}\left(e_\pm^{(i)}\right)\right\| \\ & \leq  \max_{\phi^{(i)}\in \EU^{s_i-1}} \left\|\frac{d\gamma^{(i)}}{d\phi^{(i)}}\right\| \left\| y^{(i)}-e_\pm^{(i)}\right\| \\ & <\chi_i \varepsilon.
\end{align*}
\end{proof}

Since $\mu_i>1$ and $\rho_i$ is the midpoint of the interval $\left[1, \mu_i \right]\subset \RR^+$, it follows that $\mu_i>\rho_i>1$ (by \textbf{(H4)}).

\begin{lemma} 
\label{Lemma11}
For each $i\in\{1, \ldots, N\}$, there exists $\varepsilon_2^{(i)}>0$ such that for all $\varepsilon\in \left[0, \varepsilon_2^{(i)}\right]$, we have:
$$
f_i(E_i(\varepsilon))\subset E_{i+1}(\zeta_i\varepsilon^{\rho_i}) \qquad \text{where} \qquad \zeta_i=\max_{\phi^{(i)}\in \EU^{s_i-1}} \left\|M^{(i)}\left(\phi^{(i)}\right)\right\| .
$$
\end{lemma}

\begin{proof}
We want to prove that if $\left(x^{(i+1)}, y^{(i+1)}\right) \in f_i(E_i(\varepsilon))$, then $\left(x^{(i+1)}, y^{(i+1)}\right) \in E_{i+1}(\varepsilon^{\rho_i})$. Indeed, if $\left(x^{(i+1)}, y^{(i+1)}\right) \in f_i(E_i(\varepsilon))$ then  there exists $\left(x^{(i)}, y^{(i)}\right) \in E_i(\varepsilon)$, \emph{i.e.} $\left\|x^{(i)}\right\|<\varepsilon$ such that:
$$
f_i\left(x^{(i)}, y^{(i)}\right) =\left(x^{(i+1)}, y^{(i+1)}\right).
$$ 
Indeed, by the expression (\ref{linearization2}), we get:
\begin{align*}
\left\|x^{(i+1)}\right\| &\leq  \max_{\phi^{(i)}\in \EU^{s_i-1}} \left\|M^{(i)}\left(\phi^{(i)}\right)\right\| \exp\left(-\lambda_{u_i+1}T_i\left(x^{(i)}\right)\right)\left\|y^{(i)}\right\|  \\ 
& \leq \max_{\phi^{(i)}\in \EU^{s_i-1}} \left\|M^{(i)}\left(\phi^{(i)}\right)\right\| \exp\left({T_i\left(x^{(i)}\right)}\right)^{-\lambda_{u_i+1}^{(i)} } \\
&  \leq \max_{\phi^{(i)}\in \EU^{s_i-1}} \left\|M^{(i)}\left(\phi^{(i)}\right)\right\|\left(x_1^{(i)}\right)^{\mu_i} \\
&  \leq \zeta_i\left\|x^{(i)}\right\|^{\rho_i},  
\end{align*}
and the result follows.
\end{proof}

\begin{figure}\begin{center}
\includegraphics[height=6cm]{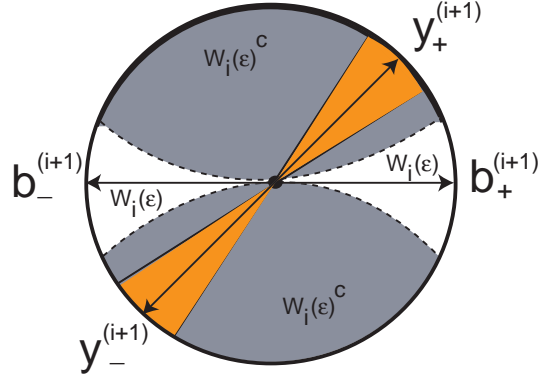}
\end{center}
\caption{\small Representation of the non-degeneracy condition (\ref{generic condition}): the global maps from $\Sigma_i^{out}$ to $\Sigma_{i+1}^{in}$ are non-degenerate and thus cusps are not mappped within other cusps. }
\label{non_nested_cusps}
\end{figure}

We wish to show that the image of most initial conditions passing through $\Sigma_i^{out}$ via $e_\pm^{(i)}$ hit $\Sigma_{i+1}^{in}$ with  nonzero component in the $x_{1,\pm}^{(i+1)}$ direction, where $x_{1,\pm}^{(i+1)}$ is the first component of 
\begin{equation}
\label{generic condition}
x_\pm^{(i+1)}=M^{(i)}\left(e_\pm^{(i)}\right)b_\pm^{(i)}.
\end{equation}
Generically, for each $i\in\{1, \ldots, N\}$, $x_{1,\pm}^{(i+1)\pmod{N^\star}} \neq 0$. This means that the global maps from $\Sigma_i^{out}$ to $\Sigma_{i+1}^{in}$ are generic and thus cusps are not mappped within other cusps.

\begin{lemma}
\label{Lemma12}
For each $i\in\{1, \ldots, N\}$, there exists $\varepsilon_3^{(i)}>0$ such that for all $\varepsilon\in \left[0, \varepsilon_3^{(i)}\right]$, we have:
$$f_i\left(\mathcal{W}_i(\varepsilon_3) \cap E_i(\varepsilon)\cap \mathcal{F}_i(\varepsilon_3)\right)\subset \mathcal{W}_{i+1}\left(\varepsilon^{2\rho_i(1-\beta_i)}\right).$$
\end{lemma}
\begin{proof}
First we may find $\varepsilon_3>0$ such that $\left(x^{(i)},y^{(i)}\right)\in E(\varepsilon_3)\cap \mathcal{F}(\varepsilon_3)$.
This means that there exists  $\left(x_1^{(i+1)}, \ldots, x_{u_i}^{(i+1)}\right)$ and $k\in \RR\backslash\{0\}$ such that:
$$
\left(x_1^{(i+1)}, \ldots, x_{u_i}^{(i+1)}\right)=
\left(\exp\left({-\lambda_{u_{i}+1}}^{(i)} T_i \left(x^{(i)}\right)\right)y_1^{(i)}, ..., \exp\left({-\lambda_{n}^{(i)} T_i\left(x^{(i)}\right)}\right)y_{s_i}^{(i)} \right)= k .b_\pm^{(i+1)}.
$$
Since $M^{(i)}\in GL(\RR^{u_{i+1}})$ and $\tau_i\left(x^{(i)}\right)$ is close to $e_\pm^{(i+1)}$, we may find $k_1>0$ such that the hypothesis of Lemma \ref{lemma com k} holds. Thus, for $\varepsilon\in \left[0, \varepsilon_3^{(i)}\right]$, it follows that:
$$
1-\left(\tau_1\left(x^{(i+1)}\right)\right)^2<k_1 ^{-\beta_i} \left\| x^{(i+1)}\right\|^{2-2\beta_i}< \varepsilon^{2\rho_i(1-\beta_i)}
$$
and we get the result.
\end{proof}

\subsection{The proof of Theorem \ref{Main}}
Let $\varepsilon_\star^{(i)}= \min\left\{\varepsilon_1^{(i)}, \varepsilon_2^{(i)}, \varepsilon_3^{(i)}\right\}>0$. For each $i\in\{1, \ldots, N^\star\}$,  using Lemmas \ref{Lemma10}, \ref{Lemma11} and \ref{Lemma12},
it follows that for all $\varepsilon \in \left[0, {\varepsilon_\star^{(i)}}\right]$, the inclusion holds, for $i \pmod{N^\star}$:
$$
f_i\left(\mathcal{W}_i\left(\varepsilon\right)\cap E_i\left(\varepsilon\right) \cap \mathcal{F}_i\left(\varepsilon\right)\right)\subset \mathcal{W}_{i+1}\left(\varepsilon^{{\rho_i}}\right)\cap E_{i+1}\left(\zeta_i\varepsilon^{{\rho_i}}\right)\cap \mathcal{F}_{i+1}(\chi_i\varepsilon^{\rho_i})
$$ 
For $n \in \NN$, define the sequence: $$\Omega_n= \mathcal{W}_{1}\left(\varepsilon^{\rho^{n-1}}\right)\cap E_{1}\left(\varepsilon^{{\rho^{n-1}}}\right)\cap \mathcal{F}_1\left( \varepsilon^{\rho^{n-1}}\right) \qquad \text{where}\quad \rho=\prod_{i=1}^{N^\star} \rho_i.
$$
By Lemma \ref{medida da cusp}, this set has asymptotically full Lesbegue measure. Therefore 
$
F(\Omega_n)\subset \Omega_{n+1},
$
implying that points in $\Omega_1$ converge uniformly to $\left(0, y_\pm^{(1)} \right)\in \Sigma^{in}_1$, under iteraction of $F$.  This finishes the proof of Theorem \ref{Main}.

\subsection*{Acknowledgments}
The author is very grateful to Valentin Afraimovich, who has sketched the proof of Theorem 2 during the Conference School  \emph{Dynamics, Bifurcations and Strange Attractors} which took place in Nizhny Novgorod - Russia.  Also special thanks to Alexander Lohse, Michael Field, Pedro Duarte, Jos\'e Pedro Gaiv\~ao and Telmo Peixe for the fruitful discussions.
\medbreak
{The author was partially supported by CMUP (UID/MAT/00144/2013), which is funded by FCT with national (MEC) and European structural funds through the programs FEDER, under the partnership agreement PT2020. AR also acknowledges financial support from Program INVESTIGADOR FCT (IF/00107/2015). Part of this work has been written during AR's stay in Nizhny Novgorod University, supported by the grant RNF 14-41-00044.}

\end{document}